\documentclass[12pt, a4paper]{article}

\usepackage{mathtools, amssymb}
\usepackage{eucal} % new S and F

\usepackage{amsthm}
\theoremstyle{plain}
\newtheorem{theorem}{Theorem}[section]
\newtheorem{proposition}[theorem]{Proposition}   
\newtheorem{lemma}[theorem]{Lemma}   
\newtheorem{corollary}[theorem]{Corollary}   
\theoremstyle{remark}

\newtheorem{remark}[theorem]{Remark}

\usepackage{tikz-cd}

\newcommand{\dlpbk}{\ar[dl, phantom, "\llcorner", very near start]}
\newcommand{\ddpbk}{\arrow[phantom]{dd}[very near start]{\rotatebox{45}{$\llcorner$}}} 

\providecommand{\kat}[1]{\text{\textbf{\textsl{#1}}}}

\newcommand{\infGrpd}{{\mathcal{S}}}

\newcommand{\simplexcat}{\boldsymbol \Delta}
\newcommand{\LIN}{\kat{LIN}}

\newcommand{\one}{\mathbf{1}}

\renewcommand{\Im}{\operatorname{Im}}

\DeclareMathOperator{\Id}{Id}
\DeclareMathOperator{\id}{id}

\newcommand{\xra}{\xrightarrow}
\newcommand{\xla}{\xleftarrow}

\newcommand{\eq}{\simeq}

\newcommand{\conv}{\ast}
\newcommand{\op}{^{\text{{\rm{op}}}}}

\newcommand{\Phieven}{\Phi_{\operatorname{even}}}
\newcommand{\Phiodd}{\Phi_{\operatorname{odd}}}
\newcommand{\Seven}{S_{\operatorname{even}}}
\newcommand{\Sodd}{S_{\operatorname{odd}}}

\providecommand{\norm}[1]{\left| {#1}\right|}

\newcommand{\upperstar}{^{\raisebox{-0.25ex}[0ex][0ex]{\(\ast\)}}}

\newcommand{\lowershriek}{_!}

\newcommand{\isopil}{\stackrel{\raisebox{0.1ex}[0ex][0ex]{\(\sim\)}}%
			{\raisebox{-0.15ex}[0.28ex]{\(\rightarrow\)}}}

\newcommand{\CC}{\mathcal{C}}

\newcommand{\Q}{\mathbb{Q}}

\def\overarrow#1{{\vec{#1}}}
\def\nondeg{\overarrow}
\newcommand{\ora}{\nondeg}

\renewcommand{\epsilon}{\varepsilon}

\usepackage{vmargin} 
\setmarginsrb{3cm}{3cm}{3cm}{2cm}{0cm}{0cm}{0cm}{1cm}

\usepackage{hyperref}
\hypersetup{
    pdftitle={Antipodes of monoidal decomposition spaces},
    pdfauthor={Louis Carlier, Joachim Kock},
    pdfcreator={},
    pdfproducer={},
    colorlinks=true,
    linkcolor=blue,
    citecolor=blue,
    urlcolor=blue,
    breaklinks=true
}

\title{Antipodes of monoidal decomposition spaces}
\author{Louis Carlier$^*$ and Joachim Kock\thanks{Supported by grants
MTM2016-80439-P  (AEI/FEDER, UE) of Spain and
  2017-SGR-1725 of Catalonia.}\\[4pt]
  \footnotesize
Departament de Matem\`atiques\\[-4pt]
\footnotesize Universitat Aut\`onoma de Barcelona}
\date{}

\begin{document}
\maketitle
%\begin{dedication}
%  To the memory of Thomas Poguntke
%\end{dedication}
\vspace{-2em}
\begin{center}
  \emph{To the memory of Thomas Poguntke}
\end{center}
\vspace{1em}

\begin{abstract}
  We introduce a notion of antipode for monoidal (complete)
  decomposition spaces, inducing a notion of weak antipode for their
  incidence bialgebras.  In the connected case, this recovers
  the usual notion of antipode in Hopf algebras.  In the non-connected
  case it expresses an inversion principle of more limited scope, but
  still sufficient to compute the M\"obius function as $\mu = \zeta
  \circ S$, just as in Hopf algebras.  At the level of decomposition
  spaces, the weak antipode takes the form of a formal difference of
  linear endofunctors $\Seven - \Sodd$, and it is a refinement
  of the general M\"obius inversion construction of
  G\'{a}lvez--Kock--Tonks, but exploiting the monoidal structure.
%   When the monoidal decomposition space is just a hereditary family of
%   poset intervals, classical antipode formulae of Schmitt are
%   recovered.
\end{abstract}

%%%%%%%%%%%%%%%%%%%%%%%%%%%%%%%%%%%%%%%%%%%%%%%%%%
\section{Introduction}
%%%%%%%%%%%%%%%%%%%%%%%%%%%%%%%%%%%%%%%%%%%%%%%%%%
\label{sec:intro}

Decomposition spaces were introduced by G\'alvez, Kock and
Tonks~\cite{GKT1,GKT2,GKT3} as a very general setting for incidence
algebras and M\"obius inversion, and independently by Dyckerhoff and
Kapranov~\cite{DK} under the name unital $2$-Segal spaces, for use in
homological algebra, representation theory and geometry.  A 
decomposition space is a simplicial $\infty$-groupoid with a property
expressing the ability to decompose objects.

It is the combinatorial perspective that concerns the present
contribution.  The line of development from the classical theory of
incidence coalgebras \cite{Joni-Rota} is summarised by regarding locally
finite posets as special instances of M\"obius 
categories~\cite{Leroux76}, which in
turn are regarded as simplicial sets via the nerve.  The crucial
observation from \cite{GKT1} is that the Segal condition (which
characterises the ability to compose as in a category) is not needed:
the decomposition-space axiom characterises instead the ability to
{\em decompose}, in a sufficiently controlled manner so as to 
allow the
construction of an incidence coalgebra.  There are countless examples in 
combinatorics of coalgebras arising from decomposition spaces but not (directly) 
from categories or posets.  The passage from simplicial sets to simplicial
groupoids is motivated by combinatorics to take into account
symmetries.  The further passage to $\infty$-groupoids is harder to
justify from combinatorics, but is the natural level of generality
from a homotopy viewpoint.  The
decomposition-space approach to incidence algebras is
{\em objective} (in the sense by Lawvere and 
Menni~\cite{Lawvere-Menni}), meaning that the constructions take place with the 
combinatorial objects 
themselves rather than with vector spaces spanned by them. In this
way, all proofs are natively `bijective'.  It is an attractive 
feature of the theory that most arguments boil down to computing pullbacks, 
by which we always mean {\em homotopy} pullbacks (i.e.~pullbacks in 
the $\infty$-category $\infGrpd$ of $\infty$-groupoids).

% This notion is interesting already for simplicial sets and
% simplicial groupoids: there are countless examples in
% combinatorics, where many coalgebras are seen to arise from
% decomposition spaces but not from categories or posets.  Two
% illustrative examples are Schmitt's chromatic Hopf algebra of
% graphs and the Butcher--Connes--Kreimer Hopf algebra of rooted
% trees.
% 
Bialgebras and Hopf algebras, rather than just coalgebras, are
obtained from {\em monoidal} decomposition spaces.  In examples from 
combinatorics, the monoidal structure is often disjoint union. It is 
characteristic for the decomposition-space approach that
the bialgebras obtained are (often filtered but) not connected in general.
In particular they are not in general Hopf.  For example, for the 
nerve of a category, all identity arrows become group-like elements in the 
coalgebra.

While the decomposition-space theory was originally modelled on incidence
coalgebras and M\"obius inversion machinery \`a la Rota, algebraic
combinatorics soon after Rota discovered the more powerful machinery
of antipodes, when available. For example, in an incidence Hopf
algebra, M\"obius inversion amounts to precomposing with the antipode $S$, 
exhibiting in particular the M\"obius function as $\mu = \zeta \circ S$. 
The work of Schmitt \cite{Schmitt:antipodes, Schmitt:IHA} was seminal to 
the change of emphasis from M\"obius inversion to antipodes. The recent
work of Aguiar and Ardila~\cite{Aguiar-Ardila} represents a
striking example of the power of antipodes.

% The present note observes that 
% % 
% % The decomposition-space machinery (in spite of its high-tech
% % homotopy flavour) may appear a bit old-fashioned in comparison,
% % focusing as it does on M\"obius inversion rather than antipodes.
% % The main fact conditioning this is of course  that many 
% % decomposition spaces are not monoidal, so that their cardinality
% % are coalgebras rather than bialgebras or Hopf algebras.
% 
% \bigskip
% 
% The bialgebras obtained have the property that degree zero is
% spanned by group-like elements, so one can obtain a Hopf algebra
% by collapsing degree zero to the multiplicative unit.
% For the incidence algebra of a M\"obius category, this amounts
% to identifying all identity arrows, irrespective of their 
% objects.
% However, G\'alvez, Kock and Tonks~\cite{GKT:faa} crucially exploited
% the spread-out (non-connected) nature of the F\`aa di Bruno 
% bialgebra (rather than its Hopf-algebra quotient).  Kock~\cite{Kock:1411.3098}
% argued that the non-connected viewpoint is both natural and
% viable in perturbative renormalisation, generalising the so-called
% BPHZ procedure from Hopf algebras to not-quite-connected 
% bialgebras.

The present note upgrades the G\'alvez--Kock--Tonks M\"obius-inversion
construction~\cite{GKT2} to the construction of a kind of antipode in any 
{\em monoidal} (complete) decomposition space.
Many of the constructions are quite similar; the main innovative idea is that 
there is a useful weaker notion of antipode for
bialgebras even if they are {\em not} Hopf.
% There is a considerable
% psychological obstruction to this endeavour: it corresponds to finding a
% endomorphism of a monoid which acts as inversion on those elements
% that are actually invertible!

We introduce this notion and  establish its main features (and 
limitations).  Briefly, for $X$ a monoidal (complete) decomposition 
space, the antipode is defined as a formal 
difference between linear endofunctors  of $\infGrpd_{/X_1}$,
$$
S := \Seven-\Sodd ,
$$
given by multiplying principal edges of nondegenerate simplices 
(cf.~p.\pageref{S} below).
It cannot quite convolution-invert the identity endofunctor, 
as a true antipode should  \cite{Sweedler},
but it can invert a modification of it, denoted $\Id'$:
$$
\Id'(f) = \begin{cases}  f & \text{ if } f \text{ nondegenerate} ,\\
\id_u & \text{ if } f \text{ degenerate}.
\end{cases}
$$
Here $u$ is the monoidal unit object, and we write $\id_u $ for $s_0 
u$.

Precisely, our main theorem (\ref{thm:antipodeformula}) 
is the inversion formula
%<HHHH
$$\boxed{
\Seven \conv \Id' \ \simeq \ e + \Sodd \conv \Id' }
$$
%HHHH>
% $$\boxed{
% \Id' \conv \Seven \ \simeq \ e + \Id' \conv \Sodd }
% $$
where $e:=\eta \circ \epsilon$ is the neutral element for convolution.
Under the finiteness conditions satisfied by {\em M\"obius} 
decomposition spaces~\cite[\S8]{GKT2}, one can take homotopy 
cardinality~\cite[\S3]{GKT:HLA} and form the difference $\norm{S} := 
\norm\Seven-\norm\Sodd$ to arrive at the nicer-looking equation in the 
$\Q$-vector-space level convolution algebra:
%<HHHH
$$
\norm S \conv \norm{\Id'} \ = \ \norm e \ = \ \norm{\Id'} \conv \norm S.
$$
%HHHH>
% $$
% \norm{\Id'} \conv \norm S \ = \ \norm e \ = \ \norm S \conv \norm{\Id'}.
% $$

The three main features justifying the weaker notion of antipode are:
\begin{enumerate}
    \item If the monoidal decomposition space is connected, so that
	its incidence bialgebra is Hopf, 
	then the homotopy cardinality of $S$ 
	is the usual antipode (cf.~Proposition~\ref{prop:classical-S}).
	(At the objective level of decomposition spaces, the construction 
	of $S$ is new also in the connected case.)
	
    \item In any case, $S$ computes the M\"obius functor as
	\[
	\mu \simeq \zeta \circ S  
	\]
	(cf.~Corollary~\ref{cor:classical-mu}).
	
	\item More generally, we establish an inversion formula for
	multiplicative functors (valued in any algebra) that send
	group-like elements to the unit (Theorem~\ref{thm:inversion}).  The
	zeta functor is an example of this.
\end{enumerate}

At the algebraic level of $\Q$-vector spaces, the weak antipode can be
seen as a lift of the true antipode from the connected quotient of the
bialgebra.
% (generalising the classical notion of reduced incidence bialgebra).
When the bialgebra comes from the nerve of a category,
this quotient is obtained by identifying all objects of the category.
Recent developments have shown the utility of avoiding this reduction,
which destroys useful information.  For example, the Fa\`a di Bruno
formula for general operads~\cite{GKT:faa}, \cite{KW} crucially
exploits the finer structure of the zeroth graded piece of the
incidence bialgebra, and in the bialgebra
version~\cite{Kock:1411.3098} of BPHZ renormalisation in perturbative
quantum field theory, the zeroth graded piece of the bialgebra of
Feynman graphs contains the terms of the Lagrangian (not visible in 
the quotient Hopf algebra usually employed).

%%%%%%%%%%%%%%%%%%%%%%%%%%%%%%%%%%%%%%%%%%%%%%%%%%
\section{Preliminaries: monoidal decomposition spaces and their 
incidence bialgebras}
%%%%%%%%%%%%%%%%%%%%%%%%%%%%%%%%%%%%%%%%%%%%%%%%%%

We assume familiarity with the basic theory of decomposition 
spaces \cite{GKT1, GKT2, GKT3}, and limit ourselves to a minimal 
background section, so as at least to establish notation.

Like in \cite{GKT1, GKT2, GKT3}, we work with
simplicial $\infty$-groupoids, since it is the natural generality of
the theory.  However, our results belong to combinatorics, where the
examples of interest are merely simplicial groupoids or even
simplicial sets (such as the nerve of a poset).  The reader may safely
substitute `groupoid' or `set' for the word `$\infty$-groupoid'
throughout.  This is the viewpoint taken in \cite{GKT:combinatorics},
which may serve as an introduction.

\paragraph{Linear functors and spans.}
Denote by $\infGrpd$ the 
$\infty$-category of $\infty$-groupoids.  
%The slices $\infGrpd_{/I}$
%play the role of vector spaces.  The objects of $\infGrpd_{/I}$ 
%are maps $A\to I$, playing the role of vectors: they are $I$-indexed 
%families
%of $\infty$-groupoids (the homotopy fibres) rather than 
%$I$-indexed families of scalars. 
A functor between slices $F: \infGrpd_{/I} \to \infGrpd_{/J}$
is called {\em linear}~\cite[\S 2]{GKT:HLA} if it is given by a span
$$
I \stackrel{p}\longleftarrow M \stackrel{q}\longrightarrow J
$$
by pullback along $p$ (denoted $p^*$) followed by composition with 
$q$ (denoted $q_!$), i.e. $F = q_! \circ p^*$.  The $\infty$-groupoid $M$ 
itself plays the role of an $(I\times J)$-indexed matrix. Crucially, 
composition of linear functors is given by taking pullback like this:
\[
\begin{tikzcd}[sep={1cm,between origins}]
  && \cdot \ar[ld] \ar[rd] \ddpbk && \\
  & M \ar[ld] \ar[rd] && N \ar[ld] \ar[rd] & \\
  I && J && K,
\end{tikzcd}
\]
the objective version of matrix multiplication.
We denote by $\LIN$ the $\infty$-category of slices of 
$\infGrpd$ and linear functors.
In suitably finite situations, one can take homotopy cardinality of
slices and linear functors to obtain vector spaces and linear maps 
(see for examples the Appendix of \cite{GKT:combinatorics}).

\paragraph{Decomposition spaces.}
A simplicial $\infty$-groupoid $X: \simplexcat\op\to\infGrpd$ is
called a {\em decomposition space} if it takes active-inert pushouts
in $\simplexcat$ to pullbacks~\cite[\S3]{GKT1}.
% The active maps in $\simplexcat$ are the 
% `end-point preserving' maps, and the inert maps are the `distance preserving'.
The precise meaning is not so important for the present purposes---it 
suffices here to say that the decomposition-space axiom precisely
ensures that the comultiplication given by the linear functor
\[
\Delta: \infGrpd_{/X_1} \ \xrightarrow{(d_2,d_0)\lowershriek \circ 
d_1\upperstar} \ \infGrpd_{/X_1\times X_1}
\]
defined by the span
\[
X_1 \stackrel{d_1}\longleftarrow
% \xleftarrow{d_1} 
X_2 \xrightarrow{(d_2,d_0)} X_1 \times X_1
\]
is up-to-coherent-homotopy coassociative as well as counital
(the counit being defined by the span $X_1 \stackrel{s_0}\leftarrow 
X_0 \to \one$) \cite[\S5]{GKT1}. Here and throughout, $\one$ 
denotes any contractible $\infty$-groupoid (the terminal 
$\infty$-groupoid), so that $\infGrpd_{/\one}\simeq \infGrpd$. The 
$\infty$-category $\infGrpd_{/X_1}$,
with $\Delta$ and $\epsilon$,
is called the {\em incidence coalgebra} of $X$.
When $X$ is locally finite~\cite[\S7]{GKT2}, one can take homotopy 
cardinality to obtain an ordinary coalgebra in $\Q$-vector 
spaces, namely $\Q_{\pi_0 X_1}$, the $\Q$-vector space spanned by path 
components of $X_1$.
Any Segal space is a decomposition space~\cite[\S3]{GKT1};
an example to keep in mind is the nerve of a small
category.

% \paragraph{CULF maps.}
%   A simplicial map $F: Y\to X$ between decomposition spaces
%   is called \emph{CULF}~\cite[\S4]{GKT1} when it is cartesian on active
%   maps (i.e.\ the naturality squares on active maps are pullbacks).
%   This can be measured on the active map $d_1$ alone: $F$ is CULF iff
%   the square
% \[
%   \begin{tikzcd}
%   	Y_1 \ar[d, "F_1"']&
% 		Y_2 \ar[d, "F_2"] \ar[l, "d_1"'] \\
% 		X_1 & X_2 \ar[l, "d_1"]
%   \end{tikzcd}
%   \]
% is a pullback.  The purpose of this notion is that CULF maps
% between decomposition spaces induce coalgebra homomorphisms~\cite[\S4]{GKT1}.

\paragraph{Complete decomposition spaces, and nondegenerate simplices.}
A decomposition space is {\em complete} (\cite[\S2]{GKT2}) when $s_0:
X_0 \to X_1$ is a monomorphism of $\infty$-groupoids (i.e.~its 
(homotopy) fibres
are either empty or contractible). (For simplicial {\em sets}, the condition 
is automatic.) This condition ensures that
there is a well-behaved notion of nondegenerate
simplices: define the space of nondegenerate $1$-simplices $\nondeg
X_1$ as the complement of the essential image of the
monomorphism $s_0: X_0 \to X_1$, so that we have
$$
X_1 \simeq X_0 + \nondeg X_1,
$$
where $+$ denotes the disjoint union.
More generally, $\nondeg X_n \subset X_n$ is characterised as the
complement of the union of the essential images of the degeneracy maps 
$s_i: X_{n-1} \to X_n$, that is
$$
\nondeg X_n = X_n \setminus \bigcup_{i=0}^{n-1} \Im(s_i) .
$$
By definition $\nondeg X_0 = X_0$.
In a complete decomposition space, an $n$-simplex is nondegenerate if and 
only if all its $n$ principal edges are nondegenerate~\cite[\S2]{GKT2}.

\paragraph{Monoidal decomposition spaces.}

 Bialgebras are obtained from decomposition spaces with a 
{\em CULF monoidal structure}~\cite[\S9]{GKT1}.  This means first of all that 
there are simplicial maps (unit and multiplication):
% or more precisely, monoid objects in 
% the monoidal $\infty$-category $(\Decomp_{\operatorname{culf}}, 
% \times, \one)$
% \cite[\S9]{GKT1}.
% (Note that $\times$, the cartesian product of simplicial spaces,
%  is not the categorical product in $\Decomp_{\operatorname{culf}}$, 
%  and that $\one$, the trivial simplicial 
% space, is
% not terminal in $\Decomp_{\operatorname{culf}}$.)  
% The monoidal structure amounts to simplicial maps
% $\mu_n : X^n \to X$ for all $n$.
% In particular we have simplicial maps
% % What it amounts
% % to is in particular a ``tensor product'' 
% % %$\otimes$ 
% % which we prefer to denote by $\mu$, with unit 
% % denoted $\eta$:
$$
\one \stackrel{\eta}\longrightarrow X 
\stackrel{\mu}\longleftarrow X \times X ,
$$
% \[
%   \begin{tikzcd}
% X_\bullet \times X_\bullet \ar[d, "\mu"'] \\
% X_\bullet
%   \end{tikzcd}
%   \]
% with a unit $\eta$
% \[
%   \begin{tikzcd}
% X_\bullet  \\
% 1_\bullet \ar[u, "\eta"]
%   \end{tikzcd}
%   \]
but with the important condition imposed that these maps should be CULF,
which is a pullback condition required expressly to ensure
that there is induced
% 
% Crucially, these are required to be CULF maps.
% 
% For general reasons, a CULF monoidal structure on $X$ induces 
a monoid structure on the 
incidence coalgebra $\infGrpd_{/X_1}$ in the $\infty$-category of 
coalgebras,
and hence altogether a
% In other words, a 
bialgebra structure on $\infGrpd_{/X_1}$.
% in 
% $\LIN$ (the $\infty$-category whose objects are slices $\infGrpd_{/I}$
% and whose morphisms are linear functors).

\paragraph{Connectedness.}

A monoidal decomposition space (or its incidence bialgebra) is called 
{\em connected} when $X_0$ is 
contractible (that is, $X_0 \simeq \one$). Usually, connectedness should 
refer to a filtration~\cite{Sweedler}. This filtration does not always exist 
for a monoidal decomposition space $X$, but it does
exist when $X$ is M\"obius (the existence of the length 
filtration is one characterisation of the M\"obius 
condition~\cite[\S8]{GKT2}).
In that case, $X_0$ spans filtration degree $0$,
so the condition $X_0 \simeq \one$ agrees with the usual notion of being
connected for filtered coalgebras (or bialgebras).
When $X$ is M\"obius and connected, its cardinality is a
connected filtered bialgebra, and therefore, by standard 
arguments~\cite{Sweedler}, a Hopf algebra.
However, many important incidence bialgebras are not connected.

\paragraph{Examples.}
  Let $X$ be the fat nerve of the category of finite sets and
  surjections.  The resulting bialgebra is the F\`aa di Bruno
  bialgebra~\cite{GKT:combinatorics}.  The zeroth graded piece 
  is spanned by the invertible 
  surjections (which are all group-like), so is not connected.  The monoidal 
  structure is disjoint union and the monoidal unit is the (identity of the) 
  empty set.
  
%   Figure out what the antipode is in this case.  $S(A_n)$ is the
%   alternating sum of all products of chains (chain meaning not
%   containing an identity arrow).

  More generally, for any reduced operad, the so-called
  two-sided bar construction $X$ is
  a monoidal (complete) decomposition space~\cite{KW}.
  % Completeness is not stated in \cite{KW}, but it is true: $s_0$ is
  % simply the monad unit, and it is mono since the operad unit has
  % no automorphisms.  (It has no automorphisms since we are talking 
  % classical operads, so there is a set of operations with symmetric 
  % group actions.  Since the only automorphisms come from the 
  % symmetric group actions, and since the operad unit has arity 1, 
  % there are no symmetries, so discrete, so the monad unit is mono.
  The
  groupoid $X_0$ is the free symmetric monoidal category on the set of
  objects of the operad. 
  (Note that $X$ is never connected.) 
  The groupoid $X_1$ is
  the free symmetric monoidal category on the action groupoid of the 
  symmetric-group actions on the set of
  operations. 
  The generalisation of the classical Fa\`a di Bruno formula to any
  operad \cite{GKT:faa,KW} (the classical case being that of the terminal reduced 
  operad) crucially exploits the typing constraints expressed by the 
  objects in $X_0$ (which are invisible in the connected quotient 
  Hopf algebra).

\section{Antipodes for monoidal complete decomposition spaces}

\paragraph{Convolution.}
Let $X$ be a monoidal decomposition space.  For $F, G :
\infGrpd_{/X_1} \to \infGrpd_{/X_1}$ two linear endofunctors, the
{\em convolution product} $F \conv G : \infGrpd_{/X_1} \to
\infGrpd_{/X_1}$ is given by first comultiplying, then composing with 
the tensor product $F \otimes G$, and finally multiplying.
% and the comultiplication, followed by multiplication of the monoidal
% structure.  
If $F$ and $G$ are given by the 
spans $X_1 \xla{} M \xra{} X_1$ and $X_1 \xla{} N \xra{} X_1$, then $F \conv G$ is given by the composite of spans
%$\infGrpd_{/X_1} \to \infGrpd_{/X_1} \otimes \infGrpd_{/X_1} \to \infGrpd_{/X_1}$
\begin{center}
  \begin{tikzcd}[row sep=large]
    X_1 & & & \\
        X_2  \ar[u, ""]
           \ar[d, ""'] 
             & M \conv N \dlpbk
                         \ar[l, ""] 
                         \ar[d, ""] 
                         \ar[ul, ""] 
                         \ar[drr, ""]& \\
        X_1 \times X_1  & M \times N \ar[l] \ar[r] & X_1 \times X_1 \ar[r, "\mu"'] & X_1.
  \end{tikzcd}
\end{center}
% This convolution is compatible with the convolution of functors
% $\infGrpd_{/X_1} \to \infGrpd$, since the multiplication is counital.

The neutral element for convolution in 
$\LIN(\infGrpd_{/X_1}, \infGrpd_{/X_1})$ is $e := \eta \circ \epsilon$. By 
composition of spans, it is given by the span
\[
    X_1 
	\stackrel{s_0}\longleftarrow
	X_0 
	\stackrel{w}\longrightarrow
% 	\xra{} \one \xra{\eta}
	X_1 ,
\]
where $w$ denotes the composite $X_0 \xra{} \one \xra{\eta} X_1$.

\paragraph{The antipode.}

Define the linear endofunctor $S_n : \infGrpd_{/X_1} \to \infGrpd_{/X_1}$ 
by the span
\begin{equation}\label{Sn}
    X_1 \stackrel g \longleftarrow 
	\ora{X}_n \xra{p} X_1 \times \ldots \times X_1 \xra{\mu_n} X_1  ,
\end{equation}
where $g$ 
% is the unique active map---it 
returns the `long edge' of a 
simplex, and $p$ returns its $n$ principal edges.

In the case $n=0$, we have $g=s_0$ and $(X_1)^0 = \one$ and $\mu_0 = \eta$, 
whence $S_0$ coincides with the neutral element:
$$
S_0 = e .
$$
Note also that the functor $S_1$ is given by the span 
\begin{equation}\label{S1}
    X_1 \stackrel i \longleftarrow 
	\ora{X}_1 \stackrel i \longrightarrow X_1.
\end{equation}

\begin{lemma}\label{antipoderecformula}
  We have
  $$
  S_n \simeq (S_1)^{*n} .
  $$
\end{lemma}

\begin{proof}
The case $n=0$ is trivial since $S_0$ is neutral.
In the convolution 
%<HHHH
$S_n \conv S_1$,
%HHHH>
% $S_1 \conv S_n$,
the main pullback is given by the Lemma 3.5 of \cite{GKT2}:
%<HHHH
\begin{center}
      \begin{tikzcd}[sep = large]
          X_1  & & &\\
          X_2 \ar[u, "d_1"] \ar[d, "{(d_2,d_0)}"'] & \ora{X}_{n+1} 
		  \ar[d, "{(d_\top, d_\bot{}^n)}"] \ar[l, "d_1\circ \cdots \circ d_{n-1}"']
		  \ar[ul, "g"'] \ar[drr, "\mu \circ p"] \dlpbk & &\\
          X_1 \times X_1  & \ora{X}_n \times \ora{X}_1 \ar[l, "g 
		  \times i"]  \ar[r, "p \times \id"'] & (\nondeg X_1)^{n} \times \nondeg X_1 \ar[r, "\mu"'] & X_1.
      \end{tikzcd}  
  \end{center}
  Commutativity of the upper triangle is precisely the face-map description of $g$.  
  The lower triangle commutes since $d_\bot{}^n$ returns itself the 
  last principal edge.
%HHHH>
%   \begin{center}
%       \begin{tikzcd}[sep = large]
%           X_1  & & &\\
%           X_2 \ar[u, "d_1"] \ar[d, "{(d_2,d_0)}"'] & \ora{X}_{1+n} 
% 		  \ar[d, "{(d_\top{}^n, d_0)}"] \ar[l, "d_2\circ \cdots \circ 
% 		  d_n"'] \ar[ul, "g"'] \ar[drr, "\mu \circ p"] \dlpbk & &\\
%           X_1 \times X_1  & \ora{X}_1 \times \ora{X}_n \ar[l, "i \times g"]  \ar[r, "\id \times p"'] & \nondeg X_1 \times (\nondeg X_1) ^{n}\ar[r, "\mu"'] & X_1.
%       \end{tikzcd}  
%   \end{center}
%   Commutativity of the upper triangle is precisely the face-map description of $g$.  
%   The lower triangle commutes since $d_\top{}^n$ returns itself the first principal edge. 
\end{proof}

Put
$$
\Seven := \sum_{n \text{ even}} S_n , \qquad
\Sodd := \sum_{n \text{ odd}} S_n .
$$
Note that the sum of linear functors is given by the 
sum (disjoint union) of the middle objects of the respesenting spans.
Hence $\Seven$ is given by the span
$$
X_1 \longleftarrow \sum_{n \text{ even}} \ora X_n 
\longrightarrow  X_1 ,
$$
where the left leg returns the long edge of a simplex, and 
the right leg returns the monoidal product of the principal edges.  
Similarly of course with $\Sodd$.

The {\em antipode} $S$ is defined as the formal difference
\label{S}
$$
S := \Seven-\Sodd .
$$
The difference cannot be formed at the objective level where 
there is no minus sign available, but it does make sense after 
taking homotopy cardinality to arrive at $\Q$-vector spaces.
For this to be meaningful, certain finiteness conditions must be 
imposed: $X$ should be {\em M\"obius}, which means locally finite and
of locally finite length, cf.~\cite[\S8]{GKT2}.
We shall continue to work with
$\Seven$ and $\Sodd$ individually.  

The idea of an antipode is that it should be convolution inverse to
the identity functor, 
%<HHHH
i.e.~$S \conv \Id$ 
%HHHH>
% i.e.~$\Id \conv S$ 
should be $\eta\circ\epsilon$.
This is not in general true for monoidal decomposition spaces.  We
show instead that $S$ inverts
% it is possible to invert
the following modified identity functor.

The linear functor $\Id' : \infGrpd_{/X_1} \to \infGrpd_{/X_1}$ is 
given by the span
\[
    X_1 \stackrel{=}\longleftarrow 
	X_0 + \ora{X}_1 \stackrel{w|i}\longrightarrow
	X_1,
\]
where $i$ is the inclusion $\ora{X}_1 \subset X_1$, and 
$w: X_0 \xra{p} \one \xra{\eta} X_1$ is the constant map with value
$\id_u$, the identity at the monoidal unit object $u$. 
In other words,
$$
\Id' \eq S_0 + S_1.
$$
On elements,
\[\Id'(f) =         
    \begin{cases}
        f     &\text{ if $f$ nondegenerate},\\
      \id_{u }, &\text{ if $f$ degenerate}.
    \end{cases}
\]

\begin{lemma}\label{lem:id'*S_n}
The linear functors $S_n$ satisfy
%<HHHH
    \[
    S_n \conv \Id' \ \eq \ S_n + S_{n+1} \ \eq \ \Id' \conv S_n.
    \]
%HHHH>
%     \[
%     \Id' \conv S_n \ \eq \ S_n + S_{n+1} \ \eq \ S_n \conv \Id'.
%     \]
\end{lemma}

\begin{proof}
Since $\Id' \eq S_0 + S_1$, the result follows from 
%<HHHH
$S_n \conv S_1 \eq S_{n+1} \eq S_1 \conv S_n$
%HHHH>
% $S_1 \conv S_n \eq S_{n+1} \eq S_n \conv S_1$
(which is a consequence of Lemma~\ref{antipoderecformula}), and 
%<HHHH
$S_n \conv S_0 \eq S_n \eq S_0 \conv S_n$
%HHHH>
% $S_0 \conv S_n \eq S_n \eq S_n \conv S_0$ 
($S_0$ is neutral for convolution).
\end{proof}

\begin{theorem}\label{thm:antipodeformula}
Given a monoidal complete decomposition space $X$, we have
explicit equivalences
%<HHHH
\[
    \Seven \conv \Id' \ \eq \ e + \Sodd \conv \Id' \qquad 
    \text{ and } \qquad  \Id' \conv \Seven \ \eq \ e + \Id' \conv \Sodd  .
\]
%HHHH>
% \[
%     \Id' \conv S_\even \ \eq \ e + \Id' \conv S_\odd \qquad 
%     \text{ and } \qquad  S_\even \conv \Id' \ \eq \ e + S_\odd \conv \Id'.
% \]
\end{theorem}

\begin{proof}
  It follows from Lemma~\ref{lem:id'*S_n} that all four
  functors are equivalent to $\sum_{n \ge 0} S_n$.
\end{proof}

\paragraph{Finiteness conditions and homotopy cardinality.}

If the  monoidal complete decomposition space $X$ is locally finite 
(meaning that $X_1$ is locally finite and $X_0 
\stackrel{s_0}\to X_1 \stackrel{d_1}\leftarrow X_2$ are finite 
maps~\cite[\S 8]{GKT2}), then we can take homotopy 
cardinality~\cite[\S3]{GKT:HLA} to obtain the 
incidence bialgebra at the $\Q$-vector space level, and obtain also
linear endomorphisms
\[
    \norm{S_n} : \Q_{\pi_0 X_1} \to \Q_{\pi_0 X_1}.
    \]
If $X$ is furthermore M\"obius, the sums involved in the definitions
of $\Seven$ and $\Sodd$ are finite, and the difference
$\norm{S} = \norm{\Seven} - \norm{\Sodd}$ is a well-defined linear
endomorphism of $\Q_{\pi_0 X_1}$, and we arrive at the following
weak antipode formula:

% If the decomposition space is M\"obius, we can take
% the cardinality of the antipode formula of theorem~\ref{thm:antipodeformula}:

\begin{proposition}
If $X$ is a M\"obius monoidal decomposition space, then we have
%<HHHH
\[
    \norm{S} \conv \norm{\Id'} = \norm{e} = \norm{\Id'} \conv \norm{S}
\] 
%HHHH>
% \[
%     \norm{\Id'} \conv \norm{S} = \norm{e} = \norm{S} \conv \norm{\Id'}
% \] 
in $\Q_{\pi_0 X_1}$, the homotopy cardinality of the incidence 
bialgebra of $X$.
\end{proposition}

\paragraph{Connectedness and the usual notion of antipode.}

We say a monoidal decomposition space is \emph{connected} if $X_0$ is
contractible.  In this situation, $X_0$ contains only the monoidal unit, 
so that the maps $w$ and $s_0$ coincide, and hence
$\Id' \simeq \Id$.  (Indeed, note that the identity endofunctor 
$\Id: \infGrpd_{/X_1} \to \infGrpd_{/X_1}$ is given by 
the span $X_1 \xla{=} X_1 \xra{=} X_1$, and that $s_0|i : X_0 + 
\nondeg X_1 \to X_1$ is an equivalence.)
We then get the following stricter inversion result, yielding the
usual notion of antipode in Hopf algebras, after 
taking homotopy cardinality:

\begin{proposition}\label{prop:classical-S}
%<HHHH
  If $X$ is a connected monoidal complete decomposition space, then 
    \[
     \Seven \conv \Id \ \eq \ e + \Sodd \conv \Id 
		\qquad \text { and } \qquad
         \Id \conv \Seven  \ \eq \ e + \Id \conv \Sodd . 
	 \]

If moreover $X$ is M\"obius, we get
    \[
    \norm{S} \conv \norm{\Id} = \norm{e} = \norm{\Id} \conv \norm{S}.
\] 
%HHHH>
%     \[
%          \Id \conv \Seven  \ \eq \ e + \Id \conv \Sodd \qquad 
% 		 \text { and } \qquad
%      \Seven \conv \Id \ \eq \ e + \Sodd \conv \Id .
% 	 \]
% 
% If moreover $X$ is M\"obius, we get
%     \[
%     \norm{\Id} \conv \norm{S} = \norm{e} = \norm{S} \conv \norm{\Id}.
% \] 
\end{proposition}

% \begin{proof}
% If we suppose $X_0 = 1$, the linear functor $id'$ is given by the span
% \[
%     X_1 \xla{=} 1 + \ora{X}_1 \xra{=} X_1,
% \]
% which is the span defining the identity linear functor. The result follows 
% from theorem~\ref{thm:antipodeformula}.
% \end{proof}

\paragraph{Relationship with classical antipode formulae.}
If $X$ is the nerve of a M\"obius category $\CC$, then the comultiplication 
formula reads
$$
\Delta(f) = \sum_{b\circ a = f} a \otimes b .
$$
The decomposition space $X$ becomes monoidal if $\CC$ is
{\em monoidal extensive} \cite[\S9]{GKT1}, meaning 
that it has a monoidal structure $(\CC,\otimes, k)$
with natural equivalences
$$
\CC/x \times \CC/y \isopil \CC/(x\otimes y), \qquad \qquad \one \isopil 
\CC/k .
$$
In combinatorics, extensive monoidal structures most often arise as 
disjoint union.

Spelling out the the general antipode formula in the case of a monoidal extensive category gives
$$
S(f) = \sum_{k\geq 0} (-1)^k 
\underset{a_i \neq \id}{\sum_{a_k \circ \cdots \circ a_1 = f}}
a_1 \cdots a_k  .
$$
When $\CC$ is just a locally finite hereditary poset (with 
intervals regarded as arrows), this is Schmitt's antipode
formula for the reduced incidence Hopf
algebra of the poset~\cite{Schmitt:antipodes}.

Schmitt's formula works more generally for hereditary families of
poset intervals, meaning classes of poset intervals that are closed
under taking subintervals and cartesian products~\cite{Schmitt:IHA}.
Our general formula
covers that case as well.  The intervals of such a family do not
necessarily come from a single poset (or even a M\"obius category).
One can prove that such a family always forms a monoidal 
decomposition space,  the most important case being 
the family of {\em all} (finite) poset intervals~\cite{GKT3}.

Other classical antipode formulae are readily extracted.
For example, from the general formula $S_{n+1} \simeq S_n \conv S_1$ (see 
Lemma~\ref{antipoderecformula}), one finds
$$
\Seven \simeq S_0 + \Sodd \conv S_1,  \qquad  \Sodd \simeq \Seven 
\conv S_1 ,
$$
whence the recursive formula
$$
S \simeq S_0 - S \conv S_1 ,
$$
valid after taking homotopy cardinality.
Spelling this out in the case of the nerve of a monoidal extensive 
M\"obius category yields the familiar formula
$$
S(f) = S_0(f) - \underset{b\neq \id}{\sum_{b\circ a = f}}
S(a) \cdot b   ,
$$
which also goes back to Schmitt~\cite{Schmitt:antipodes}, in the poset 
case.

%%%%%%%%%%%%%%%%%%%%%%%%%%%%%%%%%%%%%%%%%%%%%%%%%%
\section{Inversion in convolution algebras}
%%%%%%%%%%%%%%%%%%%%%%%%%%%%%%%%%%%%%%%%%%%%%%%%%%

\paragraph{M\"obius inversion.}

The M\"obius inversion formula \cite[\S3]{GKT2} is 
recovered easily from Theorem~\ref{thm:antipodeformula}.
Recall that the {\em zeta functor} is the linear functor
$\zeta: \infGrpd_{/X_1} \to \infGrpd$ defined by the span 
$X_1 \stackrel=\leftarrow X_1 \to \one$. 
% M\"obius inversion 
% states that $\zeta$ is invertible for the convolution product:
% there exists $\mu$ such that $\mu\conv \zeta = \epsilon$.

First we define
\[
    \Phi_n := \zeta \circ S_n.
\]
By composition of spans, $\Phi_n$ is given by
$$
X_1 \stackrel{g}\longleftarrow \nondeg X_n \longrightarrow \one
$$
in accordance with \cite{GKT2}.  We also get
$$
\Phieven := \zeta \circ \Seven  = \sum_{n \text{ even}} \Phi_n ,
\qquad 
\Phiodd := \zeta \circ \Sodd = \sum_{n \text{ odd}} \Phi_n  .
$$

The following is now an immediate consequence of 
Theorem~\ref{thm:antipodeformula}.
\begin{corollary}[\cite{GKT2} Theorem 3.8]\label{cor:classical-mu}
For a monoidal complete decomposition space, the M\"obius
inversion principle holds, expressed by the explicit equivalences
%<HHHH
\[
    \Phieven \conv \zeta \eq \epsilon + \Phiodd \conv \zeta 
	\quad \text{ and } \qquad
	\zeta \conv \Phieven \eq \epsilon + \zeta \conv \Phiodd.
\]
%HHHH>
% \[
%     \zeta \conv \Phi_\even \eq \epsilon + \zeta \conv \Phi_\odd, 
% 	\quad \text{ and } \qquad
% 	\Phi_\even \conv \zeta \eq \epsilon + \Phi_\odd \conv \zeta.
% \]
\end{corollary}

This proof is a considerable simplification compared to
the proof given in \cite{GKT2}, but note that it crucially depends on
the monoidal structure.  
The theorem of \cite{GKT2} is more general in that it works also
in the absence of a monoidal structure.

\paragraph{More general inversion.}
One advantage of the antipode over the M\"obius inversion formula
is that it gives
a uniform inversion principle, rather than just inverting the zeta
function.  At the $\Q$-vector space level, the result $\norm\mu =
\norm\zeta \circ \norm S$ is readily generalised as follows.  Let
$B_X$ denote the homotopy cardinality of the incidence bialgebra of a
monoidal M\"obius decomposition space $X$.

\begin{lemma}\label{Q-lemma}
    For any $\Q$-algebra $A$ with unit $\eta_A$, consider the 
  convolution algebra $(\operatorname{Lin}(B_X, A), \conv, \eta_A 
  \epsilon)$.  If $\phi: B_X \to A$ is multiplicative and sends
  all group-like elements to $\eta_A$, then $\phi$ is convolution
  invertible with inverse $\phi\circ S$.
\end{lemma}

\begin{proof}
  Indeed, `multiplicative' ensures that
  $\phi\circ (S \conv \Id') = (\phi\circ S) \conv (\phi\circ \Id')$,
%   $\phi\circ (\Id' \conv S) = (\phi\circ \Id') \conv (\phi\circ S)$,
  and the condition on group-like elements ensures
that $\phi\circ \Id' = \phi$ (and that $\phi \circ \eta_B = \eta_A$).
\end{proof}

The {\em connected quotient} $H_X$ is defined as $H_X:= B_X/J_X$, where 
$$J_X = \langle s_0 x - s_0 u \mid x\in X_0 \rangle ,
$$ 
which is a Hopf ideal \cite{Sweedler}
since the elements $s_0 x$ are group-like. 
(Here $u$ denotes the monoidal unit.)
It is clear that $H_X$ is
connected, hence a Hopf algebra.  Now the conditions on $\phi$ in
Lemma~\ref{Q-lemma} amount precisely to saying that $\phi$ vanishes on
the Hopf ideal $J_X$, and hence factors through
the quotient Hopf algebra $H_X$:
\[
\begin{tikzcd}
  B_X \ar[rr, "\phi"] \ar[rd] && A    \\
  & H_X \ar[ru, "\overline{\phi} "', dotted] &
  \end{tikzcd}
\]

From this perspective, the weak antipode of $B_X$ does not invert
anything that could not have been inverted with classical technology, namely
by the true antipode in $H_X$.
The point of the weak antipode is that it is defined already at the 
objective level of decomposition spaces, without the need of quotienting.
% where it is not known whether the notion of
% connected quotient makes sense.
% (An ad hoc construction was given for the Fa\`a di Bruno Hopf algebra
% as a quotient decomposition space~\cite{GKT:combinatorics}, but it is not obvious that such a
% construction is always possible).  The antipode notion we have 
% introduced for monoidal decomposition spaces is a clear-cut approach 
% to inversion, which bypasses the need for quotienting.
We shall establish the following objective version of 
Lemma~\ref{Q-lemma}.
\begin{theorem}\label{thm:inversion}
  Let $X$ be a monoidal complete decomposition space, and 
  let $A$ be a monoidal $\infty$-groupoid---this makes 
  $\infGrpd_{/A}$ an algebra in $\LIN$.  Consider the convolution 
  algebra $(\LIN(\infGrpd_{/X_1},\infGrpd_{/A}), \conv, \eta_A 
  \epsilon)$.  If a linear functor $\phi: 
  \infGrpd_{/X_1} \to \infGrpd_{/A}$ is multiplicative and contracts 
  degenerate elements, then $\phi$ is convolution
  invertible with inverse $\phi\circ S$.
\end{theorem}
The main task is to define the notions involved.
Throughout, we let $X$ denote a monoidal complete decomposition 
space, and $A$ a monoidal $\infty$-groupoid.
A linear functor $\phi: 
  \infGrpd_{/X_1} \to \infGrpd_{/A}$ given by a span
  $$
  X_1 \stackrel{u}\leftarrow F \stackrel{v}{\to} A
  $$
  is called {\em multiplicative} if it is a span of monoidal functors with 
  $u$ CULF.  This means that we have commutative diagrams
  \begin{equation}\label{eq:multiplicative}
  \begin{tikzcd}
	X_1\times X_1 \ar[d, "\mu_1"']  & F\times F \dlpbk \ar[d, "\mu_F"] 
	\ar[l, "u\times u"'] \ar[r, "v\times v"] & A\times A \ar[d, 
	"\mu_A"] \\
	X_1 & F \ar[l, "u"] \ar[r, "v"'] & A
  \end{tikzcd}
  \qquad
  \begin{tikzcd}
	\one \ar[d, "\eta_1"']  & \one \dlpbk \ar[d, "\eta_F"] 
	\ar[l, "="'] \ar[r, "="] & \one \ar[d, 
	"\eta_A"] \\
	X_1 & F \ar[l, "u"] \ar[r, "v"'] & A  .
  \end{tikzcd}
  \end{equation}
  Commutativity of the diagrams expresses of course that the functors 
  $u$ and $v$ are monoidal. 
%   (In reality all higher 
%   multiplicativities  should be accounted for too\ldots)  
  CULFness amounts to the 
  pullback conditions indicated, which are required because we need to do 
  pull-push along these squares.
%   (The CULF notion expressed by this pullback 
%   condition is at a higher level than the standard CULFness employed 
%   for simplicial maps between simplicial $\infty$-groupoids.  Here
%   the notion applies to monoidal functors between monoidal 
%   $\infty$-groupoids.  By regarding monoidal $\infty$-groupoids as
%   $1$-object $(\infty,2)$-categories, and thereby simplicial 
%   $\infty$-categories, the CULFness agrees with the simplicial 
%   notion.)

  A linear functor $\phi: 
  \infGrpd_{/X_1} \to \infGrpd_{/A}$ given by a span
  $$
  X_1 \stackrel{u}\leftarrow F \stackrel{v}{\to} A
  $$
  is said to {\em contract degenerate elements} if the following condition
  holds:
 \begin{equation}\label{eq:contracts}
   \begin{tikzcd}
	X_0\ar[d, "s_0"']  & X_0 \dlpbk \ar[d, "s_F"] 
	\ar[l, "="'] \ar[r, "p"] & \one \ar[d, 
	"\eta_A"] \\
	X_1 & F \ar[l, "u"] \ar[r, "v"'] & A  .
  \end{tikzcd}
  \end{equation}
  Two conditions are expressed by this: the first is that $u$ pulled 
  back along $s_0$ gives the identity map.  (The map $s_F$ is defined by 
  this pullback.)  The second condition says that $v\circ s_F$ 
  factors through the unit.  Altogether, the conditions express the
  idea of mapping all degenerate elements to the unit object of $A$.
  
  \begin{lemma}\label{contract=unital}
	If $\phi$ contracts degenerate elements 
	(Equation~\eqref{eq:contracts}), then it is unital 
	(Equation~\eqref{eq:multiplicative} RHS).
\end{lemma}

\begin{proof}
  In the diagram
  \begin{center}
      \begin{tikzcd}
          \one \ar[d, "\eta_0"'] 
            & \one \dlpbk \ar[l, "="'] \ar[r, "="] \ar[d, "\eta_0"] 
            & \one \ar[d, "="]\\
          X_0 \ar[d, "s_0"'] 
            & X_0 \ar[d, "s_F"] \ar[r, "p"] \ar[l, "="'] \dlpbk
            & \one \ar[d, "\eta_A"] \\
          X_1 & F \ar[l,"u"] \ar[r,"v"'] & A ,
      \end{tikzcd}
  \end{center}
  the bottom squares are \eqref{eq:contracts}, and the outline 
  diagram is \eqref{eq:multiplicative} RHS, since
  the composite vertical arrows are $\eta_1$, $\eta_F$, and $\eta_A$.
\end{proof}

\begin{lemma}\label{phimult}
   If a linear functor $\phi: \infGrpd_{/X_1} \to \infGrpd_{/A}$
is multiplicative, then $\phi \circ -$ distributes over convolution.
Precisely, for any linear endofunctors $\alpha, \beta: \infGrpd_{/X_1} 
\to \infGrpd_{/X_1}$, we have
   \[\phi \circ (\alpha \conv \beta) \eq (\phi \circ \alpha) \conv (\phi \circ \beta).\]
\end{lemma}
\noindent Note that $\conv$ on the left refers to convolution of 
endofunctors, while $\conv$ on the right refer to convolution in 
$\LIN(\infGrpd_{/X_1},\infGrpd_{/A})$.

\begin{proof}
The left-hand side $\phi \circ (\alpha \conv \beta)$ is
computed by the pullbacks 
  
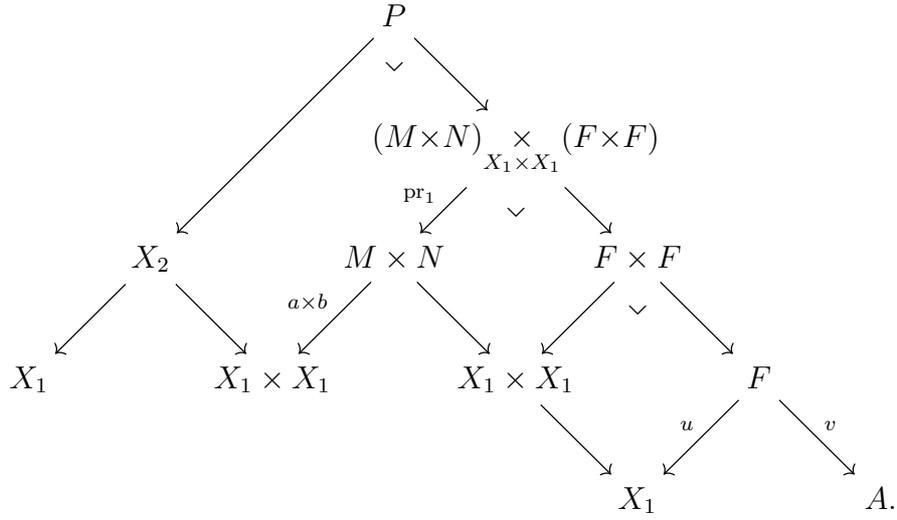
\begin{center}
   \begin{tikzcd}[sep={1.6cm,between origins}]
     & & & P \ar[ddll, ""'] \ar[dr,""] \ddpbk & & & & \\
     & & & &   (M {\times} N) {\underset{X_1 \times X_1}\times} (F 
	 {\times} F) \ar[dl, "\operatorname{pr}_1"'] \ar[dr, ""] \ddpbk & & & \\
     & X_2 \ar[dl] \ar[dr] & &M \times N \ar[dl, "a \times b"'] \ar[dr] & & F \times F \ar[dl] \ar[dr, ""'] \ddpbk & &  \\
     X_1 & & X_1 \times X_1 & & X_1 \times X_1 \ar[dr] & & F \ar[dl, "u"'] \ar[dr, "v"] & \\
     & & & & & X_1 & & A.
 \end{tikzcd}   
  \end{center}
  The right-hand side $(\phi \circ \alpha) \conv (\phi \circ
\beta)$ is computed by the pullback
 
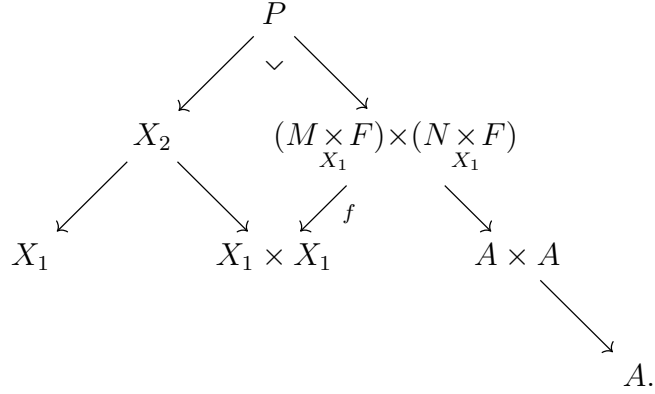
\begin{center}
\begin{tikzcd}[sep={1.6cm,between origins}]
  && P  \ar[ld] \ar[rd] \ddpbk &&& \\
  & X_2 \ar[ld, ""'] \ar[rd, ""]  &&
{(M{\underset{X_1}\times} F)} {\times}  (N{\underset{X_1}\times} F)
\ar[ld,near start, "f"] \ar[rd]  & & \\
X_1 && X_1 \times X_1 & & A \times A \ar[rd] & \\
&&&& & A.
\end{tikzcd}
 \end{center}  
Here $f$ is the map $(a\circ \operatorname{pr}_1) \times
(b\circ \operatorname{pr}_1)$. 
These two composed spans agree since clearly
\[(M \times N) \underset{X_1 \times X_1}\times (F \times F)
\simeq
(M\underset{X_1}\times F) {\times}  (N\underset{X_1}\times F)
\]
 (and $f \simeq (a\times b) \circ \operatorname{pr}_1$).
\end{proof}

\begin{lemma}\label{phiidprime}
    If a linear functor $\phi: \infGrpd_{/X_1} \to \infGrpd_{/A}$
    contracts degenerate elements, then we have
  \[
    \phi \circ \Id' \eq \phi.
  \]
\end{lemma}

\begin{proof}
  Let $\omega$ denote the endofunctor defined by the span $X_1 
  \stackrel{s_0} \longleftarrow X_0 \stackrel{s_0}\longrightarrow 
  X_1$.
  Since $\Id' = S_0 + S_1$ and $\Id = \omega + S_1$, it is enough to
  establish
  $$
  \phi \circ S_0 \simeq \phi \circ \omega .
  $$
  The left-hand side $\phi \circ S_0$ is computed by the pullbacks
  
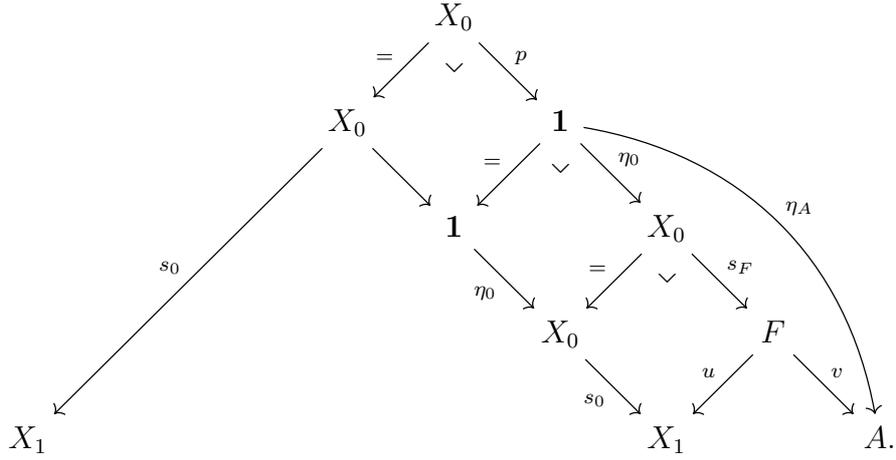
\begin{center}
   \begin{tikzcd}[sep={1.4cm,between origins}]
     & & &     & X_0\ar[dl, "="'] \ar[dr,"p"] \ddpbk & & & & \\
     & & & X_0 \ar[dddlll, "s_0"'] \ar[dr] &   & \one \ar[dl, "="'] \ar[dr, "\eta_0"] \ar[dddrrr, bend left=35, "\eta_A"] \ddpbk & & & \\
     & & & & \one  \ar[dr, "\eta_0"'] & & X_0 \ar[dl, "="'] \ar[dr, "s_F"] \ddpbk & & \\
     & & & & & X_0 \ar[dr, "s_0"'] & & F \ar[dl, "u"'] \ar[dr, "v"] & \\
     X_1 & & & & & & X_1 & & A.
 \end{tikzcd}    
 \end{center}   
  
  The right-hand side $\phi \circ \omega$ is computed by the pullback
  
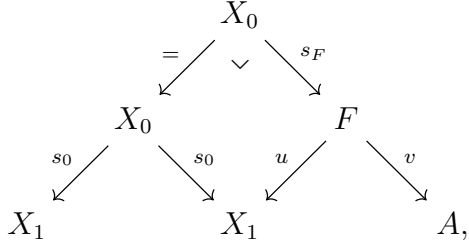
\begin{center}
   \begin{tikzcd}[sep={1.4cm,between origins}]
      & & X_0 \ar[dl, "="'] \ar[dr, "s_F"] \ddpbk & & \\
      & X_0 \ar[dl, "s_0"'] \ar[dr, "s_0"] & & F \ar[dl, "u"'] \ar[dr, "v"] & \\
      X_1 & & X_1 & & A,
    \end{tikzcd}    
  \end{center}   
  and the composite $v \circ s_F$ is again $\eta_A \circ p$ by hypothesis.
\end{proof}

\begin{proof}[Proof of Theorem~\ref{thm:inversion}]
  We need to show that $\phi\circ S$ is convolution inverse to $\phi$.
  With the preparations made, this is now direct:
%<HHHH
  \[
  (\phi\circ S) \conv \phi
  \;\stackrel{\ref{phiidprime}}\simeq\;
  (\phi\circ S) \conv (\phi\circ \Id')
  \;\stackrel{\ref{phimult}}\simeq\;
  \phi \circ (S \conv \Id') 
  \;\stackrel{\ref{thm:antipodeformula}}\simeq\;
  \phi\circ \eta_1 \circ \epsilon
  \;\stackrel{\ref{contract=unital}}\simeq\;
  \eta_A \circ \epsilon   .
  \]
%HHHH>
%   \[
%   \phi \conv (\phi\circ S) 
%   \;\stackrel{\ref{phiidprime}}\simeq\;
%   (\phi\circ \Id') \conv (\phi\circ S) 
%   \;\stackrel{\ref{phimult}}\simeq\;
%   \phi \circ ( \Id'\conv S ) 
%   \;\stackrel{\ref{thm:antipodeformula}}\simeq\;
%   \phi\circ \eta_1 \circ \epsilon
%   \;\stackrel{\ref{contract=unital}}\simeq\;
%   \eta_A \circ \epsilon   .
%   \]
\end{proof}

\begin{remark}
  The more general M\"obius inversion principle of Lemma~\ref{Q-lemma}
  and Theorem~\ref{thm:inversion} is of interest for two reasons.
  Firstly, in the connected case, the general M\"obius inversion
  principle, which we here derived from the antipode, but which can be
  formulated without reference to $S$, is actually {\em equivalent} to
  the existence of the antipode. Indeed, if one takes $A$ to
  be $X_1$ itself (so that at the cardinality level one uses $B_X$ as
  the algebra $A$), and takes $\phi$ to be the identity map, then the
  resulting inverse is the antipode.
  
  Secondly, the extra generality serves to highlight the tight analogy
  between M\"obius inversion and abstract Hopf-algebraic
  renormalisation in perturbative quantum field theory, as explained
  in \cite{Kock:1809.00941}.  The $\phi$ are then the (regularised)
  Feynman rules (which are inherently multiplicative, and can be
  arranged to send group-like elements to $1$).  In this generality,
  the passage from M\"obius inversion to renormalisation consists
  just in adding a Rota--Baxter operator to the formulae (see
  \cite{Kock:1809.00941} for details).  (The 
  result is then no longer an inverse but rather a counter-term.)
 \end{remark}

% \bibliographystyle{abbrv}
% \bibliography{biblio}

\end{document}